\documentclass{amsart}
\usepackage{tikz}
\usepackage{xcolor}
\usepackage{enumerate}
\usepackage{amssymb}
\date{\today}
\title[Vizing's Theorem and Tuza's Conjecture]{Maximal $k$-Edge-Colorable Subgraphs,\\ Vizing's Theorem,
and Tuza's Conjecture}
\author{Gregory J.~Puleo}
\address{Coordinated Science Lab, University of Illinois at Urbana-Champaign. Now at Department of Mathematics and Statistics, Auburn University.}
\tikzstyle{vertex}=[inner sep = 0pt, minimum width=4pt, fill=black, shape=circle]
\tikzstyle{squarevert}=[inner sep = 0pt, minimum width=4pt, minimum height=4pt, fill=white, shape=rectangle, draw=black, thick]
\newcommand{\gpoint}[2]{\node[style=vertex, label=#1:$#2$]}

\newcommand{\bpoint}[1]{\gpoint{below}{#1}}
\newcommand{\apoint}[1]{\gpoint{above}{#1}}

\newcommand{\GDelta}{G_{\Delta}}
\newcommand{\GDmu}{G_{\Delta,\mu}}
\newcommand{\Dmu}{\Delta^\mu}

\newcommand{\sey}{\mathcal{S}}

\newcommand{\join}{\vee}

\newcommand{\sizeof}[1]{\left\lvert{#1}\right\rvert}

\newcommand{\st}{\colon\,}
\newcommand{\apk}{\alpha'_k}
\newcommand{\Gstar}{G^*}
\newcommand{\dF}[1]{d^F(#1)}
\newcommand{\dM}[1]{d_M(#1)}
\newcommand{\dG}[1]{d_G(#1)}
\newcommand{\Fk}[1]{F(#1)}
\newcommand{\Uk}[1]{U(#1)}
\newtheorem{proposition}{Proposition}[section]
\newtheorem{conjecture}[proposition]{Conjecture}
\newtheorem{theorem}[proposition]{Theorem}
\newtheorem{lemma}[proposition]{Lemma}
\newtheorem{observation}[proposition]{Observation}
\newtheorem{corollary}[proposition]{Corollary}
\theoremstyle{definition}
\newtheorem{definition}[proposition]{Definition}
\theoremstyle{remark}

\begin{document}
\begin{abstract}
  We prove that if $M$ is a maximal $k$-edge-colorable subgraph of a
  multigraph $G$ and if $F = \{v \in V(G) \st d_M(v) \leq k-\mu(v)\}$,
  then $d_F(v) \leq d_M(v)$ for all $v \in V(G)$ with $d_M(v) < k$.
  (When $G$ is a simple graph, the set $F$ is just the set of vertices
  having degree less than $k$ in $M$.) This implies Vizing's Theorem as well
  as a special case of Tuza's Conjecture on packing and covering of
  triangles. A more detailed version of our result also implies
  Vizing's Adjacency Lemma for simple graphs.
\end{abstract}
\maketitle
\section{Introduction}
A \emph{proper $k$-edge-coloring} of a multigraph $G$ without loops is
a function $\psi : E(G) \to [k]$ such that $\psi(e) \neq \psi(f)$
whenever $e$ and $f$ are distinct edges sharing an endpoint (or both
endpoints), where $[k] = \{1, \ldots, k\}$. A graph is
\emph{$k$-edge-colorable} if it admits a proper $k$-edge-coloring. We
will tacitly assume in the rest of this paper that all multigraphs
under consideration are loopless.

A fundamental theorem concerning edge-coloring is Vizing's
Theorem~\cite{vizing}.  Given a multigraph $G$, we write $\mu_G(v,w)$
for the number of edges joining two vertices $v$ and $w$, and we write
$\mu_G(v)$ for $\max_{w \in V(G)} \mu_G(v,w)$. When the graph $G$ is understood, we
omit the subscripts.  We also write $\Delta(G)$ for the maximum degree
of $G$ and $\mu(G)$ for $\max_{v \in V(G)} \mu(v)$. Vizing's~Theorem
can then be stated as follows:
\begin{theorem}[Vizing~\cite{vizing}]\label{thm:real-vizing}
  If $G$ is a multigraph and $k \geq \Delta(G) + \mu(G)$, then $G$ is
  $k$-edge-colorable.
\end{theorem}
Following the notation of \cite{stiebitz}, let
$\Dmu(G) = \max_{v \in V(G)}[d(v) + \mu(v)]$. Since
$\Dmu(G) \leq \Delta(G) + \mu(G)$ for any multigraph $G$, and since
this inequality is sometimes strict, the following theorem of
Ore~\cite{ore-fourcolor} strengthens Theorem~\ref{thm:real-vizing}.
\begin{theorem}[Ore~\cite{ore-fourcolor}]\label{thm:ore}
  If $G$ is a multigraph and $k \geq \Dmu(G)$, then $G$
  is $k$-edge-colorable.
\end{theorem}
In this paper, we prove the following generalization of Theorem~\ref{thm:ore}. Here, when $F \subset V(G)$, we write $d_F(v)$ for $\sum_{w
  \in F}\mu(v,w)$, and when $M \subset E(G)$, we write $d_M(v)$ for
the total number of $M$-edges incident to $v$.
\begin{theorem}\label{thm:simple-main}
  Let $G$ be a multigraph, let $k \geq 1$, and let $M$ be a maximal
  $k$-edge-colorable subgraph of $G$.  If $F = \{v \in V(G) \st d_M(v)
  \leq k - \mu(v)\}$, then for every $v \in V(G)$ with $d_M(v) < k$,
  we have $d_F(v) \leq d_M(v)$.
\end{theorem}
Theorem~\ref{thm:simple-main} is easiest to understand in the case of
simple graphs, where $\mu(v)=1$ for all $v$. In this case, $F$ is just
the set of all vertices with fewer than $k$ colors present on the
incident edges, that is, the set of all vertices missing at least one
color\footnote{The letter $F$ is meant to evoke the word
``de\textbf{f}icient'', the letter $D$ being unavailable since it is
used in a different context in this paper.}.

It is also instructive to consider Theorem~\ref{thm:simple-main} in the
cases $k=1$ and $k=2$. Since a maximal matching in a graph $G$ is just a maximal
$1$-edge-colorable subgraph of $G$, the $k=1$ case of
Theorem~\ref{thm:simple-main} just states the observation that the set
of vertices left uncovered by a maximal matching is independent.

In the case $k=2$, we can observe that in a maximal $2$-edge-colorable
subgraph $M \subset G$, every component of $M$ is an even cycle or a
path (possibly a $1$-vertex path), and the vertices of $F$ are the endpoints of the path
components. Theorem~\ref{thm:simple-main} then states that $G[F]$
induces a graph consisting of a matching together with possibly some
isolates, where all vertices isolated in $M$ are also isolated in
$G[F]$. This conclusion is not difficult to prove directly, as the
maximality of $M$ implies that the only $G$-edges among the vertices
of $F$ are edges that join the the endpoints of the same path, if this
would yield an odd cycle.

For $k > 2$, no simple characterization of $k$-edge colorable graphs
is known, so a direct appeal to the structure of $M$ is not
possible. However, Theorem~\ref{thm:simple-main} still yields the
following corollary.
\begin{corollary}\label{cor:maxdelta}
  If $G$ is a simple graph, $M$ is a maximal $k$-edge-colorable
  subgraph of $G$, and $F$ is the set of vertices with fewer than $k$
  incident $M$-edges, then $\Delta(G[F]) \leq k-1$.
\end{corollary}

To see that Theorem~\ref{thm:simple-main} implies
Theorem~\ref{thm:ore}, observe that if $k \geq \Dmu(G)$ and $M$ is a
maximal $k$-edge-colorable subgraph of $G$, then $F = V(G)$, so
Theorem~\ref{thm:simple-main} states that $d_M(v) \geq d_G(v)$ for
every vertex $v$. As $M$ is a subgraph of $G$, this implies $M=G$, so
that $G$ is $k$-edge-colorable.  In Section~\ref{sec:forest}, we show
that Theorem~\ref{thm:simple-main} also implies a multigraph version
of a strengthening of Vizing's Theorem due to Lovasz and
Plummer~\cite{lovasz-plummer} and to Berge and
Fournier~\cite{berge-fournier}.

In order to prove Theorem~\ref{thm:simple-main}, we actually prove a
more technical version of the theorem, with a somewhat stronger
conclusion. This version of Theorem~\ref{thm:simple-main} is similar
to Vizing's Adjacency Lemma, and we explore the connection in more
detail in Section~\ref{sec:VAL}.
\begin{definition}
  Given a multigraph $G$, a subgraph $M \subset G$, and an integer $k \geq 1$,
  for each $v \in V(G)$ we define vertex sets $\Fk{v}$ and $\Uk{v}$ by
  \begin{align*}
    \Fk{v} &= \{w \in N(v) \st d_M(w) \leq k - \mu_G(v,w)\}, \\
    \Uk{v} &= \{w \in \Fk{v} \st \mu_M(v,w) < \mu_G(v,w)\}.
  \end{align*}
  We also write $\dF{v}$ for $d_{\Fk{v}}(v)$, that is, $\dF{v}$ is the total number
  of edges from $v$ to the vertices in $\Fk{v}$. The superscript here is meant to emphasize
  that the $F$ in this notation is a set depending on $v$, rather than being a fixed set
  as in Theorem~\ref{thm:simple-main}. Figure~\ref{fig:fu-illustration} illustrates
  the definition of $\Fk{v}$ and $\Uk{v}$.
\end{definition}
\begin{figure}
  \centering
  \begin{tikzpicture}[xscale=2]
  \bpoint{v} (v) at (0cm, 0cm) {};
  \apoint{} (x0) at (-1cm, 1cm) {};
  \apoint{} (x1) at (0cm, 1cm) {};
  \apoint{} (x2) at (1cm, 1cm) {};
  \apoint{} (x3) at (-1cm, -1cm) {};
  \apoint{} (x4) at (1cm, -1cm) {};
  \draw (v) -- (x0);
  \draw (v) .. controls ++(135:.3cm) and ++(-135:.3cm) .. (x1);
  \draw[very thick] (v) .. controls ++(45:.3cm) and ++(-45:.3cm) .. (x1);
  \draw[very thick] (v) -- (x2);
  \draw[very thick] (v) .. controls ++(0:.5cm) and ++(-90:.5cm) .. (x2);
  \draw (v) -- (x3);
  \draw[very thick] (x3) -- ++(-.3cm, 0cm);
  \draw[very thick] (x3) -- ++(-.3cm, -.3cm);
  \draw[very thick] (x3) -- ++(.3cm, -.3cm);
  \draw[very thick] (x3) -- ++(.3cm, 0cm);
  \draw[very thick] (x0) -- ++(-.3cm, .3cm);
  \draw[very thick] (x0) -- ++(0cm, .3cm);
  \draw[very thick] (x0) -- ++(.3cm, .3cm);
  \draw (x1) -- ++(-.3cm, .3cm);
  \draw[very thick] (x1) -- ++(.3cm, .3cm);
  \draw (x0) -- (x4);
  \draw (v) .. controls ++(0:.5cm) and ++(90:.5cm) .. (x4);
  \draw[very thick] (x4) -- ++(-.3cm, -.3cm);
  \draw[very thick] (x4) -- ++(0cm, -.3cm);
  \draw[very thick] (x4) -- ++(.3cm, -.3cm);
  \draw (-1.2cm, 1.2cm) rectangle (.2cm, .8cm); \node[anchor=east] at (-1.2cm, 1cm) {$\Uk{v}$};
  \draw (-1.8cm, 1.4cm) rectangle (1.2cm, .6cm); \node[anchor=west] at (1.2cm, 1cm) {$\Fk{v}$};
  \end{tikzpicture}
  \caption{Illustration of $\Fk{v}$ and $\Uk{v}$ for a vertex $v$, in the case $k=4$. Thick edges denote edges in $M$;
  vertices have no incident edges aside from those pictured.}
  \label{fig:fu-illustration}
\end{figure}
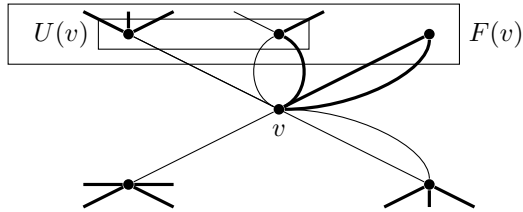
\begin{theorem}\label{thm:main}
  Let $G$ be a multigraph, let $k \geq 1$, and let $M$ be a maximal
  $k$-edge-colorable subgraph of $G$. For every $v \in V(G)$ with
  $d_M(v) < k$, we have
 \[ \dF{v} \leq d_M(v) - \sum_{w \in \Uk{v}}(k - d_M(w) - \mu_G(v,w)). \]
\end{theorem}
Note that since $\Uk{v} \subset \Fk{v}$ by definition, we have
$d_M(w) \leq k - \mu_G(v,w)$ for all $w \in \Uk{v}$, so that each term
$k - d_M(w) - \mu_G(v,w)$ in the above sum is nonnegative.
Furthermore, when $F_0$ is the set defined in
Theorem~\ref{thm:simple-main}, we see that
$(N(v) \cap F_0) \subset \Fk{v}$ for all $v \in V(G)$. Thus,
Theorem~\ref{thm:main} indeed strengthens
Theorem~\ref{thm:simple-main}.

We now consider a conjecture of Tuza regarding packing and covering
of triangles.
\begin{definition}
  Given a graph $G$, let $\tau(G)$ denote the minimum size of an edge
  set $X$ such that $G-X$ is triangle-free, and let $\nu(G)$ denote
  the maximum size of a set of pairwise edge-disjoint triangles in
  $G$.
\end{definition}
It is easy to show that $\nu(G) \leq \tau(G) \leq 3\nu(G)$: if $\sey$
is a largest set of pairwise edge-disjoint triangles, then to make $G$
triangle-free we must delete at least one edge from each triangle of
$\sey$, and on the other hand deleting all edges contained in
triangles of $\sey$ will always make $G$ triangle-free. Tuza
conjectured a stronger upper bound.
\begin{conjecture}[Tuza's Conjecture~\cite{TuzaProc,Tuza}]
  $\tau(G) \leq 2\nu(G)$ for all graphs $G$.
\end{conjecture}
Tuza's Conjecture is sharp, if true; as observed by Tuza~\cite{Tuza},
equality in the upper bound is achieved by any graph whose blocks are
all isomorphic to $K_4$, among other examples. The best general upper
bound on $\tau(G)$ in terms of $\nu(G)$ is due to
Haxell~\cite{Haxell}, who showed that $\tau(G) \leq 2.87\nu(G)$ for
all graphs $G$. Tuza's Conjecture has been studied by many 
authors, who proved the conjecture for special classes of graphs
\cite{SashaK4,Krivelevich,Puleo,LBT,LBT-perfect} or studied various
fractional relaxations of the conjecture
\cite{CDMMS,HaxellRodl,SashaStability,Krivelevich}.

A major theme of the author's previous work on Tuza's
Conjecture~\cite{Puleo} is to reduce questions about triangle packings
to questions about matchings, since matchings are very well
understood. To further pursue this idea, we study the conjecture on graphs
of the form $I_k \join H$, where $I_k$ is an independent set of size
$k$, $H$ is a triangle-free graph, and the \emph{join} $G_1 \join G_2$
of two graphs $G_1$ and $G_2$ is obtained from the disjoint union of
$G_1$ and $G_2$ by adding all possible edges between $V(G_1)$ and
$V(G_2)$. Each triangle of $I_k \join H$ consists of an edge in $H$
together with a vertex of $I_k$; thus, triangle packings in $I_k \join
H$ correspond to partial $k$-edge-colorings of $H$. In
Section~\ref{sec:tuza}, we prove a similar correspondence for edge
sets whose deletion results in a triangle-free graph, and we use
Theorem~\ref{thm:simple-main} to prove the following special case of
Tuza's~Conjecture.
\begin{theorem}\label{thm:trifree-tuza}
  If $H$ is triangle-free and $k \geq 1$, then $\tau(I_k \join H) \leq 2\nu(I_k \join H)$.
\end{theorem}
A similar idea, restricted to $k=1$, appears in \cite{SashaStability}
and \cite{CDMMS}, where $H$ is taken to be a triangle-free Ramsey
graph with small independence number, and graphs of the form $I_1
\join H$ are used as sharpness examples for upper bounds on $\tau(G)$.

The rest of the paper is structured as follows. In
Section~\ref{sec:proof} we prove Theorem~\ref{thm:main}, which implies
Theorem~\ref{thm:simple-main}. In Section~\ref{sec:forest} we use
Theorem~\ref{thm:simple-main} to prove a stronger version of Theorem~\ref{thm:ore}. In Section~\ref{sec:tuza} we use
Theorem~\ref{thm:simple-main} to prove Theorem~\ref{thm:trifree-tuza};
we also state a conjecture arising naturally from the tools used in
this proof. Finally, in Section~\ref{sec:VAL} we show that
Theorem~\ref{thm:main} implies the simple-graph case of Vizing's
Adjacency Lemma.
\section{Proof of Theorem~\ref{thm:main}}
\label{sec:proof}
Let $y \in V(G)$ with $d_M(y) < k$ be given.  We will show that
$\dF{y} \leq d_M(y) - \sum_{z \in \Uk{y}}(k - d_M(z) - \mu_G(y,z))$. Fix a proper $k$-edge-coloring $\psi$ of $M$.
We use a family of auxiliary multidigraphs defined by
Kostochka~\cite{sasha-vizing}.
\begin{definition}
  For distinct $w,z \in V(G)$, let $\psi(w,z)$ be the set of colors
  used by $\psi$ on edges joining $w$ and $z$. (If there are no edges
  joining $w$ and $z$, then $\psi(w,z) = \emptyset$.)  For each $w \in
  V(G)$, let $\psi(w)$ be the set of all colors used on edges incident
  to $w$, and let $O(w) = [k] \setminus \psi(w)$.
\end{definition}
Observe that $O(y)$ is nonempty, since $d_M(y) < k$.
\begin{definition}
  For each $u \in \Uk{y}$, let $H_u$ be the multidigraph with vertex set
  $N_M(y) \cup \{u\}$, where the number of arcs $\mu_{H_u}(w,z)$ from
  $w$ to $z$ is given by
  \[ \mu_{H_u}(w,z) = \sizeof{O(w) \cap \psi(y,z)}. \]
\end{definition}
Kostochka proved the following useful properties of the digraphs $H_u$,
under the hypothesis that $M+yu$ has no $k$-edge-coloring. By the maximality
of $M$, this hypothesis holds in our context as well. Recall that $v$ is \emph{reachable}
from $u$ in the digraph $H_u$ if $H_u$ contains a directed path from $u$ to $v$.
\begin{lemma}[Kostochka~\cite{sasha-vizing}]\label{lem:sasha-oy}
  If $v$ is reachable from $u$ in $H_u$, then $O(v) \cap O(y) = \emptyset$.
\end{lemma}
\begin{definition}
  When $\alpha$ and $\beta$ are colors, an
  \emph{$[\alpha,\beta]$}-path is a path in $M$ whose edges (under the
  coloring $\psi$) are alternately colored $\alpha$ and $\beta$.  For
  $v,w \in V(M)$, an \emph{$[\alpha,\beta](v,w)$-path} is an
  $[\alpha, \beta]$-path whose endpoints are $v$ and $w$.
\end{definition}
\begin{lemma}[Kostochka~\cite{sasha-vizing}]\label{lem:sasha-path}
  If $v$ is reachable from $u$ in $H_u$, then for each $\alpha \in
  O(y)$ and each $\beta \in O(v)$, there is an
  $[\alpha,\beta](y,v)$-path.
\end{lemma}
Kostochka~\cite{sasha-vizing} focused on studying graphs for which the
maximal $k$-edge-colorable subgraph $M$ consists of all edges of $G$
except a single edge $yu$, and therefore focused on a single digraph
$H_u$. In contrast, we work with graphs for which $M$ may be much
smaller, and therefore wish to work with many of these digraphs
simultaneously, which is facilitated by the following definitions.
\begin{definition}
  Say that $z \in N_M(y)$ is \emph{remote} if for all $u \in \Uk{y}$,
  the vertex $z$ is not reachable from $u$ in $H_u$. For each $w \in
  \Uk{y} \cup N_M(y) \cup \{y\}$, define $C(w)$ as follows: if $w$ is remote, then
  $C(w) = \psi(y,w)$, and otherwise $C(w) = O(w)$.  (In particular,
  $C(y) = O(y)$.)
\end{definition}
\begin{observation}\label{obs:remote}
  Since, by definition, $\Fk{y} \subset N_G(y)$, we have
  $\Fk{y} \subset \Uk{y} \cup N_M(y)$, so that $C(z)$ is defined for
  every $z \in \Fk{y}$.  If $z$ is remote and $z \in \Fk{y}$, then in
  particular, $z \notin \Uk{y}$, so $\mu_M(y,z)=\mu_G(y,z)$.
\end{observation}
Our next lemma strengthens Claim~3 of Kostochka~\cite{sasha-vizing}. It
can also be viewed as generalizing Lemma~1 of Andersen~\cite{andersen}
to the context of more than one uncolored edge.
\begin{lemma}\label{lem:disjoint}
  For all distinct $w,z \in N_M(y) \cup \{y\}$, we have $C(w) \cap C(z) = \emptyset$. 
\end{lemma}
\begin{proof}
  If $w=y$ and $z$ is remote, then
  $C(y) \cap C(z) = O(y) \cap \psi(y,z) = \emptyset$.  If $w=y$ and
  $z$ is not remote, then Lemma~\ref{lem:sasha-oy} implies that
  $C(y) \cap C(z) = O(y) \cap O(z) = \emptyset$. Hence we may assume that
  $y \notin \{w,z\}$.

  If $w,z$ are both remote, then since $\psi$ is a proper coloring, we
  see that $C(w) \cap C(z) = \psi(y,w) \cap \psi(y,z) = \emptyset$.

  If $z$ is remote and $w$ is not remote, then there is some $u \in \Uk{y}$
  such that $w$ is reachable from $u$ in $H_u$ while $z$ is not
  reachable from $u$, so that $H_u$ has no arc $wz$. By the
  definition of $H_u$, this implies that
  $C(w) \cap C(z) = O(w) \cap \psi(y,z) = \emptyset$.

  Thus, we may assume that neither $w$ nor $z$ is remote. Let $\alpha
  \in O(y)$ and suppose that there is some $\beta \in O(w) \cap O(z)$.
  Let $P$ be the unique maximal $[\alpha,\beta]$-path starting at $y$.
  Lemma~\ref{lem:sasha-path} implies that both $w$ and $z$ are the other
  endpoint of $P$, which is impossible. Hence $C(w) \cap C(z) = O(w) \cap O(z) = \emptyset$.
\end{proof}
Now we complete the proof of Theorem~\ref{thm:main}.  First we argue
that $\sizeof{C(z)} \geq \mu_G(z, y)$ for all $z \in \Fk{y}$. If $z$
is remote, then by Observation~\ref{obs:remote}, all edges from $z$ to $y$ are colored, hence
$\sizeof{C(z)} = \mu_G(z,y)$.  If $z$ is not remote, then since $z \in
\Fk{y}$, we have
\[ \sizeof{C(z)} = \sizeof{O(z)} = k - d_M(z) \geq \mu_G(z,y). \]
Lemma~\ref{lem:disjoint} implies that $\sum_{z \in
  \Fk{y}}\sizeof{C(z)} \leq k - \sizeof{C(y)}$, so we have
\begin{align*}
  \dF{y} &= \sum_{z \in \Fk{y}}\mu_G(z,y) \\
  &\leq \sum_{z \in \Fk{y}}\sizeof{C(z)} - \sum_{z \in \Uk{y}}(\sizeof{C(z)} - \mu_G(z,y)) \\
  &\leq k - \sizeof{C(y)} - \sum_{z \in \Uk{y}}(\sizeof{C(z)} - \mu_G(z,y)) \\
  &= k - \sizeof{O(y)} - \sum_{z \in \Uk{y}}(\sizeof{C(z)} - \mu_G(z,y)) \\
  &= d_M(y) - \sum_{z \in \Uk{y}}(\sizeof{C(z)} - \mu_G(z,y)).
\end{align*}
Now for $z \in \Uk{y}$ we have
\[ \sizeof{C(z)} = \sizeof{O(z)} = k - d_M(z), \]
so we conclude that
\[ \dF{y} \leq d_M(y) - \sum_{z \in \Uk{y}}[k - d_M(z) -
\mu_G(z,y)].\]
\section{Forests of Maximum Degree}\label{sec:forest}
Following the notation of Anstee and Griggs~\cite{anstee-griggs}, 
given a multigraph $G$, we define $\GDelta$ to be the subgraph of $G$
induced by the vertices of maximum degree. We also define $\GDmu$
to be the subgraph of $G$ induced by the vertices which have
\emph{both} maximum degree and maximum multiplicity. (Possibly
no such vertices exist, as occurs when $\Dmu(G) < \Delta(G) + \mu(G)$;
in this case, we consider $\GDmu$ to be a graph with no vertices
and no edges.)

The following theorems give conditions on $\GDelta$ or $\GDmu$ which imply the
stronger claim that $G$ can be properly edge-colored with fewer colors than Theorem~\ref{thm:real-vizing}
would require.
\begin{theorem}[Berge--Fournier~\cite{berge-fournier}]\label{thm:BF}
  If $k \geq \Delta(G) + \mu(G) - 1$ and $\GDmu$ has no edges, then $G$ is $k$-edge-colorable.
\end{theorem}
\begin{theorem}[Lovasz--Plummer~\cite{lovasz-plummer} and Berge--Fournier~\cite{berge-fournier}]\label{thm:LP}
  If $\mu(G)=1$, $k \geq \Delta(G)$, and $G_{\Delta}$ is a forest, then $G$ is $k$-edge-colorable.
\end{theorem}

In this section, we use Theorem~\ref{thm:main} to prove the following
common generalization of Theorem~\ref{thm:BF} and
Theorem~\ref{thm:LP}. Let $\Gstar$ be the subgraph of $G$ induced
by all vertices $v$ such that $d(v) + \mu(v) = \Dmu(G)$. Note that $\Gstar$
only differs from $\GDmu$ when $\Dmu(G) < \Delta(G) + \mu(G)$. For such
graphs, Theorem~\ref{thm:ore} implies that $G$ is $(\Delta(G)+\mu(G)-1)$-edge-colorable without 
\emph{any} further restriction on the graph structure, while the following theorem implies
that $G$ can be edge-colored with fewer colors if it satisfies certain restrictions on $\Gstar$.
\begin{theorem}\label{thm:newforest}
  If $k \geq \Dmu(G)-1$ and $\Gstar$ has no cycle of length greater than $2$, then $G$ is $k$-edge-colorable.
\end{theorem}
Equivalently, the hypothesis of Theorem~\ref{thm:newforest} is that
merging parallel edges in $\Gstar$ should yield a forest. As in the
proof of Theorem~\ref{thm:ore} from Theorem~\ref{thm:simple-main}, our
proof will not make explicit reference to any particular
edge-coloring, only to maximal $k$-edge-colorable subgraphs of $G$.
\begin{proof}
  Fixing a value of $k$, we use induction on $\sizeof{E(\Gstar)}$,
  with base case when $\Gstar$ has no edges or when
  $k \geq \Dmu(G)$. If $k \geq \Dmu(G)$ then
  Theorem~\ref{thm:ore} immediately implies that $G$ is $k$-edge-colorable.
  Thus, we may assume that $k = \Dmu(G)-1$.

  Suppose that $\Gstar$ has no edges. By Theorem~\ref{thm:ore}, $G -
  V(\Gstar)$ is $k$-edge-colorable. Among all $k$-edge-colorable
  subgraphs of $G$ containing $E(G-V(\Gstar))$, choose $M$ to be
  maximal. The only possible edges in $E(G)-E(M)$ are edges incident
  to vertices of $\Gstar$.

  Let $F = \{v \in V(G) \st d_M(v) \leq k-\mu_G(v)\}$, as in
  Theorem~\ref{thm:simple-main}.  For all $v \in V(G)-V(\Gstar)$, we have
  $d_G(v) + \mu_G(v) < \Dmu(G)$, hence
  \[ d_M(v) \leq d_G(v) \leq k-\mu_G(v), \]
  and so $V(G) - V(\Gstar) \subset F$.

  Now consider any $v \in V(\Gstar)$.  Since $\Gstar$ has no edges, we
  have $d_G(v) = d_F(v)$, so if $d_M(v) < d_G(v)$, then
  Theorem~\ref{thm:simple-main} yields the contradiction
  $d_F(v) > d_M(v) \geq d_F(v)$.  Thus, $E(G)-E(M)$ has no edge
  incident to any vertex of $\Gstar$. By the choice of $M$, this
  implies that $M=G$. This proves the claim when $\Gstar$ has
  no edges.

  Now suppose that $\Gstar$ contains some edges. 
  Let $v$ be a ``leaf
  vertex'' in $\GDmu$, that is, choose a vertex $v$ that has
  exactly one neighbor $w$ in $\GDmu$, possibly with
  $\mu(v,w) > 1$. Let $M$ be the graph obtained from $G$ by removing
  one copy of the edge $vw$.

  We claim that $M$ is $k$-edge-colorable.  If $\Dmu(M) < \Dmu(G)$,
  then Theorem~\ref{thm:ore} implies that $M$ is
  $k$-edge-colorable. On the other hand, if $\Dmu(M) = \Dmu(G)$, then
  $V(M^*) = V(\Gstar) - \{v, w\}$, so the induction hypothesis
  implies that $M$ is $k$-edge-colorable.

  Now if $G$ is not $k$-edge-colorable, then $M$ is a maximal $k$-edge-colorable
  subgraph of $G$. Let $F = \{v \in V(G) \st d_M(v) \leq k-\mu(v)\}$, as in
  Theorem~\ref{thm:simple-main}. As before, $V(G) - V(\Gstar) \subset F$.
  Furthermore, as $v$ and $w$ are both incident to the uncolored edge $vw$,
  we have $v,w \in F$. Thus all neighbors of $v$ lie in $F$, since $v$
  has no other neighbor in $\Gstar$. Theorem~\ref{thm:simple-main} now
  yields the contradiction $d_F(v) = d_G(v) > d_M(v) \geq d_F(v)$.
  It follows that $G$ is $k$-edge-colorable.
\end{proof}
\section{Tuza's Conjecture}\label{sec:tuza}
In this section, we consider only simple graphs.
\begin{definition}[Fink--Jacobson~\cite{FinkJacobson1,FinkJacobson2}]
  For positive integers $k$, a vertex set $D \subset V(G)$ is
  \emph{$k$-dependent} if the induced subgraph $G[D]$ has maximum
  degree at most $k-1$. A vertex set $D$ is \emph{$k$-dominating} if
  $\sizeof{N(v) \cap D} \geq k$ for all $v \in V(G) - D$.
\end{definition}
\begin{definition}
  For any set $D \subset V(G)$ and any $k \geq 1$, define
  $\phi_k(D) = k\sizeof{D} - \sizeof{E(G[D])}$, and define $\phi_k(G)
  = \max_{D \subset V(G)}\phi_k(D)$. A \emph{$k$-optimal set} is a
  $k$-dependent set achieving this maximum value of $\phi_k$.
\end{definition}
The notation $\phi_k(D)$ is borrowed from the survey paper
\cite{DomSurvey}, but the function $\phi_k$ appears to have first been
studied by Favaron~\cite{Favaron}, who proved that every $k$-optimal
set is $k$-dominating, thereby answering a question posed by Fink and
Jacobson~\cite{FinkJacobson1,FinkJacobson2}. While \cite{Favaron,DomSurvey}
considered sets which maximize $\phi_k$ only over $k$-dependent vertex sets,
rather than considering a maximum over all vertex sets as we do here, the
following lemma shows that this does not change the maximum value achieved.
\begin{lemma}\label{lem:subset}
  If $G$ is a graph and $T \subset V(G)$, then for any $k \geq 1$, there is a
  $k$-dependent subset $D \subset T$ such that $\phi_k(D) \geq \phi_k(T)$.
  In particular, every graph has a $k$-optimal set.
\end{lemma}
\begin{proof}
  If $T$ is not $k$-dependent, then there is some $v \in T$ with $d_T(v) \geq k$;
  now $\phi_k(T-v) \geq \phi_k(T)$. Repeatedly removing such vertices yields
  the desired $k$-dependent subset.
\end{proof}
\begin{definition}
  For a graph $G$ and $k \geq 1$, let $\alpha'_k(G)$ denote
  the largest number of edges in a $k$-edge-colorable subgraph of $G$.
\end{definition}
\begin{theorem}\label{thm:tuzaconnection}
  If $H$ is triangle-free, then
  \begin{align*}
    \nu(I_k \join H) &= \alpha'_k(H),\text{ and} \\
    \tau(I_k \join H) &= k\sizeof{V(H)} - \phi_k(H).
  \end{align*}
\end{theorem}
\begin{proof}
  Let $G = I_k \join H$.  We first show that $\nu(G) = \apk(H)$. Let
  $\sey$ be a maximum set of edge disjoint triangles in $G$.  For each
  $v \in I_k$, let $\sey_v = \{T \in \sey \st v \in T\}$. Since each
  triangle in $G$ consists of exactly one vertex of $I_k$ together
  with an edge in $H$, we can write $\sey$ as the disjoint union $\sey
  = \bigcup_{v \in I_k}\sey_v$. Since the triangles in $\sey_v$ are
  edge-disjoint, no two triangles in $\sey_v$ can share a common
  vertex $w \in V(H)$: if this were the case, they would intersect in
  the edge $vw$. Hence the edges of $\sey_v$ that lie in $H$ form a
  matching $M_v$ in $H$. Since the triangles in $\sey$ are
  edge-disjoint, it follows that the matchings $M_v$ are pairwise
  disjoint, so $\bigcup_{v \in I_k}M_v$ is a $k$-edge-colorable
  subgraph of $H$ having size $\nu(G)$. Therefore, $\nu(G) \leq
  \apk(H)$.

  On the other hand, if $H_0$ is a maximum $k$-edge-colorable subgraph of $H$,
  then we can write $E(H_0) = M_1 \cup \cdots \cup M_k$, where each $M_i$ is
  a matching. Let $v_1, \ldots, v_k$ be the vertices of $I_k$, and for $i \in [k]$,
  let $\sey_i = \{v_iwz \st wz \in M_i\}$. Now $\bigcup_{i \in [k]}\sey_i$ is a family
  of $\apk(H)$ pairwise edge-disjoint triangles in $G$, so $\nu(G) \geq \apk(H)$.

  Next we show that $\tau(G) = k\sizeof{V(H)} - \phi_k(H)$. Let $D$ be a $k$-optimal
  subset of $V(H)$, and define an edge set $X$ by
  \[ X = E(G[D]) \cup \{vw \st \text{$v \in I_k$ and $w \in
    V(H)-D$}\}. \] Clearly, $\sizeof{X} = \sizeof{E(G[D])} +
  k(\sizeof{V(H)}-\sizeof{D})$, which rearranges to $\sizeof{X} =
  k\sizeof{V(H)} - \phi_k(H)$, since $D$ is $k$-optimal.  We claim that
  $G-X$ is triangle-free. Let $T$ be any triangle in $G$; we may write
  $T=uvw$, where $uw \in E(H)$ and $v \in I_k$. If $u \notin D$, then
  $vu \in X$, and likewise for $w$. On the other hand, if $u,w \in D$,
  then $uw \in X$. Hence $G-X$ is triangle-free, so $\tau(G) \leq k\sizeof{V(H)} - \phi_k(H)$.

  Conversely, let $X$ be a minimum edge set such that $G-X$ is
  triangle-free. For each $v \in I_k$, let $C_v = \{w \in V(H) \st vw
  \in X\}$.  We transform $X$ so that all the sets $C_v$ are equal:
  pick $v^* \in I_k$ to minimize $\sizeof{C_{v^*}}$, and define $X_1$
  by
  \[ X_1 = (X \cap E(H)) \cup \{vw \st w \in C_{v^*}\}. \]
  Now $G-X_1$ is triangle-free: if $vwz$ is a triangle in $G-X_1$, then $v^*wz$
  is a triangle in $G-X$, contradicting the assumption that $G-X$ is
  triangle-free. Furthermore, by the minimality of $\sizeof{C_{v^*}}$, we
  have $\sizeof{X_1} \leq \sizeof{X}$.

  Therefore, $\sizeof{X_1} = \sizeof{X} = \tau(G)$. Let $D = V(H) - C_{v^*}$.
  Since $G-X_1$ is triangle-free, we have $E(G[D]) \subset X_1$, and so
  \[ \sizeof{X_1} \geq \sizeof{E(G[D])} + k\sizeof{V(H) - D} =
  k\sizeof{V(H)} - \phi_k(D). \] As $\phi_k(D) \leq
  \phi_k(H)$, we conclude that
  \[ \tau(G) = \sizeof{X_1} \geq k\sizeof{V(H)} - \phi_k(H).\qedhere\]
\end{proof}
\begin{theorem}\label{thm:alphi}
  For any graph $G$ and any $k \geq 1$, $2\alpha'_k(G) \geq k\sizeof{V(G)} - \phi_k(G)$.
\end{theorem}
\begin{proof}
  Let $M$ be a maximal $k$-edge-colorable subgraph of $G$, and let
  $F = \{v \in V(G) \st d_M(v) < k\}$. By the degree-sum formula
  and Theorem~\ref{thm:simple-main}, we have
  \begin{align*}
    2\sizeof{E(M)} &= \sum_{v \in V(G)}d_M(v) \\
    &= k\sizeof{V(G)} - k\sizeof{F} + \sum_{v \in F}d_M(v) \\
    &\geq k \sizeof{V(G)} - k\sizeof{F} + \sum_{v \in F}d_F(v) \\
    &\geq k\sizeof{V(G)} - k\sizeof{F} + \sizeof{E(G[F])}\\
    &= k\sizeof{V(G)} - \phi_k(F) \geq k\sizeof{V(G)} - \phi_k(G). \qedhere
  \end{align*}
\end{proof}
\begin{corollary}
  If $H$ is triangle-free and $k \geq 1$, then $\tau(I_k \join H) \leq 2\nu(I_k \join H)$.
\end{corollary}
The problem of finding lower bounds on $\alpha'_k(G)$ has been studied
by several authors
\cite{maxcolor1,maxcolor2,maxcolor3,maxcolor4,maxcolor5}, usually with
the goal of finding approximation algorithms. While
Theorem~\ref{thm:alphi} gives a lower bound on $\alpha'_k(G)$, the
same bound applies even for ``small'' maximal $k$-edge-colorable
subgraphs of $G$, and therefore typically will not be sharp.

We close this section with a conjecture concerning $k$-optimal sets
which would furnish an alternative proof of Theorem~\ref{thm:alphi}.
First consider the case $k=1$. A $1$-optimal set in a graph $G$
is just a maximum independent set of $G$, and a $1$-edge-colorable
subgraph is just a matching. The following theorem of Berge therefore relates
$1$-optimal sets and $1$-edge-colorable subgraphs of $G$.
\begin{theorem}[Berge~\cite{Berge}]\label{thm:berge}
  An independent set $D$ is a maximum independent set if and only if,
  for every independent set $T$ disjoint from $D$, there is a matching
  of $T$ into $D$.
\end{theorem}
\begin{corollary}\label{cor:covermatching}
  If $D$ is a maximum independent set in a graph $G$, then $G$ has
  a matching that covers every vertex of $V(G)-D$.
\end{corollary}
\begin{proof}
  Let $M_1$ be a maximal matching in $V(G)-D$, and let $S$ be the set
  of vertices in $V(G)-D$ not saturated by $M_1$. Since $M_1$ is a
  maximal matching, $S$ is an independent set. By
  Theorem~\ref{thm:berge}, there is a matching $M_2$ of $S$ into
  $D$. Thus, $M_1 \cup M_2$ is a matching that covers $V(G)-D$.
\end{proof}
Corollary~\ref{cor:covermatching} suggests the following generalization
to higher values of $k$.
\begin{conjecture}\label{coj:kcover}
  If $D$ is a $k$-optimal set in a graph $G$, then $G$ has
  a $k$-edge-colorable subgraph $M$ such that $d_M(v) = k$
  for all $v \in V(G)-D$.
\end{conjecture}
Conjecture~\ref{coj:kcover} would be, in some sense, a converse to
Corollary~\ref{cor:maxdelta}, which states that for every maximal
$k$-edge-colorable subgraph $M$, the set of vertices having $M$-degree
less than $k$ is a $k$-dependent set.  We also remark that Lovasz's
$(g,f)$-factor theorem~\cite{Lovasz-gf} implies the following weaker
version of Conjecture~\ref{coj:kcover}.
\begin{proposition}
  If $D$ is a $k$-optimal set in a graph $G$, then $G$ has a subgraph
  $M$ of maximum degree at most $k$ such that $d_M(v) = k$ for all $v
  \in V(G)-D$.
\end{proposition}
\section{Theorem~\ref{thm:main} and Vizing's Adjacency Lemma}\label{sec:VAL}
Say that an edge $e$ in a multigraph $G$ is \emph{critical} if
$\chi'(G-e) < \chi'(G)$.  Vizing's Adjacency Lemma~\cite{vizing} was
originally formulated for simple graphs $G$ with
$\chi'(G) = \Delta(G)+1$ such that every edge is critical.  The
following multigraph formulation of Vizing's Adjacency
Lemma~\cite{vizing} was given by
Andersen~\cite{andersen}.
\begin{lemma}[Andersen~\cite{andersen}]\label{lem:VAL}
  Let $G$ be a graph with $\chi'(G) = \max_{v \in V(G)}[d(v) + \mu(v)]$,
  and let $xy$ be a critical edge of $G$. If $t = d(x) + \mu(x,y)$,
  then $y$ has at least $\chi'(G) - t + 1$ neighbors $z$ other
  than $x$ such that $d(z) + \mu(y,z) = \chi'(G)$.
\end{lemma}
We show that the simple graph case of Lemma~\ref{lem:VAL} follows from
Theorem~\ref{thm:main}. However, the fully general multigraph case
requires a more detailed analysis, and in that case it seems we can do
no better than rewording the proof given by Andersen~\cite{andersen};
thus, we consider only simple graphs.
\begin{proof}[Proof of Lemma~\ref{lem:VAL} for simple graphs]
  In the simple graph case, we have $\chi'(G) = \Delta(G)+1$.
  Let $k = \chi'(G) - 1 = \Delta(G)$ and let $M = G-xy$. By hypothesis, $M$ is
  a maximal $k$-edge-colorable subgraph of $G$. For $z \in N_G(y) - \{x\}$,
  we have $d_M(z) = d_G(z)$, so if $d(z) < \Delta(G)$, then
  $z \in \Fk{y}$. Furthermore, $x \in \Fk{y}$, since
  \[ d_M(x) + 1 = d_G(x) \leq k. \] It
  follows that if $z \in N_G(y) - \Fk{y}$, then $z$ is a neighbor of
  $y$ other than $x$ such that $d(z) + \mu(y,z) = \chi'(G)$.  Thus,
  the desired claim follows if we can show that $\sizeof{N_G(y) -
    \Fk{y}} \geq \chi'(G)-t+1$.

  Since $x \in \Fk{y}$ and since $xy$ is the only uncolored edge in the graph,
  we have $\Uk{y} = \{x\}$. Therefore, the conclusion of Theorem~\ref{thm:main}
  yields
  \[ \dF{y} \leq \dM{y} - (k - \dM{x} - 1) = \dM{y} - (k - d_G(x)). \]
  Since $\dG{y} = \dM{y} +1$, this rearranges to
  \[ \dG{y} - \dF{y} \geq k - d_G(x) + 1 = \chi'(G) - t + 1. \]
  Since $G$ is a simple graph, we have $\sizeof{N_G(y) - \Fk{y}} = \dG{y} - \dF{y}$,
  so we are done.
\end{proof}
\section{Acknowldgments}
The author acknowledges support from the IC Postdoctoral Fellowship. The author
also thanks the anonymous referees for their careful reading of the paper
and for their helpful suggestions which improved both the presentation
and the historical accuracy of the paper.
\bibliographystyle{amsplain}
\bibliography{biblio}
\end{document}